\documentclass{amsart}

\usepackage{amsmath,amsfonts}
\usepackage[english]{babel}

\newtheorem{theorem}{Theorem}[section]
\newtheorem{lemma}[theorem]{Lemma}

\newtheorem{corollary}[theorem]{Corollary}
\newtheorem{remark}[theorem]{Remark}

\newtheorem{problem}[theorem]{Problem}

\begin{document}
\def\ss{{\mathbb S}^2}
\def\rr{{\mathbb R}^2}
\def\css{{\rm cl}_{{\mathbb S}^2}}
\def\crr{{\rm cl}_{{\mathbb R}^2}}

\title{On $C$-embedded subspaces of the Sorgenfrey plane}{On $C$-embedded subspaces of the Sorgenfrey plane}

\author{Olena Karlova}{O. Karlova}

\begin{abstract}
 We prove that every $C^*$-embedded subset of $\ss$ is a hereditarily Baire subspace of $\mathbb R^2$.
 We also show that for a subspace $E\subseteq\{(x,-x):x\in\mathbb R\}$ of the Sorgenfrey plane $\mathbb S^2$ the following conditions are equivalent: (i) $E$ is $C$-embedded in $\mathbb S^2$;  (ii) $E$ is $C^*$-embedded in $\mathbb S^2$; (iii) $E$ is a countable $G_\delta$-subspace of $\rr$ and (iv) $E$ is a countable functionally closed subspace of $\ss$.
\end{abstract}

\section{Introduction} Recall that a subset $A$ of a topological space $X$ is called {\it functionally open} ({\it functionally closed}) {\it in $X$} if there exists a continuous function $f:X\to [0,1]$ such that $A=f^{-1}((0,1])$ ($A=f^{-1}(0)$).  Sets $A$ and $B$ are {\it completely separated in $X$} if there exists a continuous function $f:X\to [0,1]$ such that $A\subseteq f^{-1}(0)$ and $B\subseteq f^{-1}(1)$.

A subspace $E$ of a topological space $X$ is
\begin{itemize}
  \item {\it $C$-embedded ($C^*$-embedded) in $X$} if every (bounded) continuous function $f:E\to \mathbb R$ can be continuously extended on $X$;

  \item  {\it $z$-embedded in $X$} if every functionally closed set in $E$ is the restriction of a functionally closed set in $X$ to $E$;

  \item {\it well-embedded in $X$}~\cite{HoshinaYamazaki} if $E$ is completely separated from any functionally closed set of $X$ disjoint from $E$.
\end{itemize}

 Clearly, every $C$-embedded subspace of $X$ is $C^*$-embedded in $X$. The converse in not true. Indeed, if  $E=\mathbb N$ and $X=\beta\mathbb N$, then $E$ is $C^*$-embedded in $X$ (see \cite[3.6.3]{Eng-eng}), but the function $f:E\to \mathbb R$, $f(x)=x$  for every $x\in E$, does not extend to a continuous function $f:X\to\mathbb R$.

 A space $X$ has {\it the property $(C^*=C)$} \cite{Ohta} if every closed $C^*$-embedded subset of $X$ is $C$-embedded in $X$.
 The classical Tietze-Urysohn Extension Theorem says that if $X$ is a normal space, then every closed subset of $X$ is $C^*$-embedded and $X$ has the property  $(C^*=C)$. Moreover, a space  $X$ is normal if and only if every its closed subset is $z$-embedded (see \cite[Proposition 3.7]{KCMUC}).

The following theorem was proved by Blair and Hager in \cite[Corollary 3.6]{BlairHager}.

\begin{theorem}\label{th:Blair-Hager}
A subset $E$ of a topological space $X$ is $C$-embedded in $X$ if and only if $E$ is $z$-embedded and  well-embedded in $X$.
\end{theorem}

A space $X$ is said to be {\it $\delta$-normally separated}~\cite{Mack} if every closed subset of $X$ is well-embedded in $X$. The class of $\delta$-normally separated spaces includes all normal spaces and all countably compact spaces. Theorem~\ref{th:Blair-Hager} implies the following result.

\begin{corollary}
  Every $\delta$-normally separated space has the property $(C^*=C)$.
\end{corollary}

According to \cite{Tanaka} every $C^*$-embedded subspace of a completely regular first countable space is closed. The following problem is still open:

\begin{problem}\cite{OpenProblems2}
  Does there exist a first countable completely regular space without property  $(C^*=C)$?
\end{problem}

H.~Ohta in \cite{Ohta} proved that the Niemytzki plane has the property $(C^*=C)$ and asked does the Sorgenfrey plane $\mathbb S^2$ (i.e., the square of the Sorgenfrey line $\mathbb S$) have the property $(C^*=C)$?

In the given paper we obtain some necessary conditions on a set $E\subseteq \ss$ to be $C^*$-embedded. We prove that every $C^*$-embedded subset of $\ss$ is a hereditarily Baire subspace of $\mathbb R^2$. We also characterize $C$- and $C^*$-embedded subspaces of the anti-diagonal $\mathbb D=\{(x,-x):x\in\mathbb R\}$ of $\ss$. Namely, we prove that for a subspace $E\subseteq \mathbb D$ of $\mathbb S^2$ the following conditions are equivalent: (i) $E$ is $C$-embedded in $\mathbb S^2$;  (ii) $E$ is $C^*$-embedded in $\mathbb S^2$; (iii) $E$ is a countable $G_\delta$-subspace of $\rr$ and (iv) $E$ is a countable functionally closed subspace of $\ss$.

\section{Every finite power of the Sorgenfrey line is a hereditarily  $\alpha$-favorable space}

Recall the definition of the Choquet game on a topological space $X$ between two players $\alpha$ and $\beta$. Player  $\beta$ goes first and chooses a nonempty open subset $U_0$ of $X$. Player $\alpha$ chooses a nonempty open subset $V_1$ of $X$ such that $V_1\subseteq U_0$. Following this player $\beta$ must select another nonempty open subset $U_1\subseteq V_1$ of $X$ and $\alpha$ must select a nonempty open subset $V_2\subseteq U_1$.  Acting in this way, the players $\alpha$ and $\beta$ obtain sequences of nonempty open sets $(U_n)_{n=0}^\infty$ and $(V_n)_{n=1}^\infty$ such that $U_{n-1}\subseteq V_n\subseteq U_n$ for every $n\in\mathbb N$. The player $\alpha$ wins if $\bigcap\limits_{n=1}^\infty V_n\ne\emptyset$. Otherwise, the player $\beta$ wins.
If there exists a rule (a strategy) such that $\alpha$ wins if he plays according to this rule, then $X$ is called {\it $\alpha$-favorable}.  Respectively, $X$ is called {\it $\beta$-unfavorable} if the player $\beta$ has no winning strategy. Clearly, every $\alpha$-favorable space $X$ is  $\beta$-unfavorable. Moreover, it is known~\cite{S-R} that a topological space $X$ is Baire if and only if it is  $\beta$-unfavorable in the Choquet game.

If $A$ is a subspace of a topological space $X$, then $\overline A$ and ${\rm int} A$ mean the closure and the interior of $A$ in $X$, respectively.

\begin{lemma}\label{union_of_alpha_favorable}
  Let $X=\bigcup\limits_{k=1}^n X_k$, where $X_k$ is an $\alpha$-favorable subspace of $X$ for every $k=1,\dots,n$. Then $X$ is an $\alpha$-favorable space.
\end{lemma}

\begin{proof}
  We prove the lemma for $n=2$. Let $G=G_1\cup G_2$, where $G_i={\rm int}\overline{X_i}$, $i=1,2$. We notice that for every $i=1,2$ the space $\overline{X_i}$ is $\alpha$-favorable, since it contains dense $\alpha$-favorable subspace. Then $G_i$ is $\alpha$-favorable as an open subspace of the $\alpha$-favorable space $X_i$. It is easy to see that the union $G$ of two open $\alpha$-favorable subspaces is an $\alpha$-favorable space. Therefore, $X$ is $\alpha$-favorable, since $G$ is dense in $X$.
\end{proof}

Let $p=(x,y)\in\mathbb R^2$ and $\varepsilon>0$. We write
\begin{gather*}
B[p;\varepsilon)=[x,x+\varepsilon)\times [y,y+\varepsilon), \\
B(p;\varepsilon)=(x-\varepsilon,x+\varepsilon)\times (y-\varepsilon,y+\varepsilon).
\end{gather*}

If $A\subseteq \ss$ then the symbol $\css A$ ($\crr A$) means the closure of $A$ in the space $\ss$ ($\rr$).

We say that a space $X$ is {\it hereditarily  $\alpha$-favorable} if every its closed subspace is  $\alpha$-favorable.

\begin{theorem}\label{th:Sorgenfey-alpha-favorable}
For every $n\in\mathbb N$  the space  $\mathbb S^n$ is hereditarily $\alpha$-favorable.
\end{theorem}

\begin{proof}
  Let $n=1$ and $\emptyset\ne F\subseteq \mathbb S$. Assume that $\beta$ chose a nonempty open in $F$ set $U_0=[a_0,b_0)\cap F$, $a_0\in F$. If  $U_0$ has an isolated point $x$ in $\mathbb S$, then $\alpha$ chooses $V_1=\{x\}$ and wins. Otherwise, $\alpha$ put $V_1=[a_0,c_0)\cap F$, where $c_0\in (a_0,b_0)\cap F$ and $c_0-a_0<1$. Now let $U_1=[a_1,b_1)\cap F\subseteq V_1$ be the second turn of $\beta$ such that $a_1\in F$ and the set $(a_1,b_1)\cap F$ has no isolated points in $\mathbb S$. Then there exists $c_1\in (a_1,b_1)\cap F$ such that $c_1-a_1<\frac 12$. Let $V_2=[a_1,c_1)\cap F$. Repeating this process, we obtain sequences $(U_m)_{m=0}^\infty$, $(V_m)_{m=1}^\infty$ of open subsets of $F$ and sequences of points $(a_m)_{m=0}^\infty$, $(b_m)_{m=0}^\infty$ and $(c_m)_{m=1}^\infty$ such that $[a_m,b_m)\supseteq [a_m,c_m)\supseteq [a_{m+1},b_{m+1})$, $c_m-a_m<\frac{1}{m+1}$, $c_m\in F$, $U_m=[a_m,b_m)\cap F$ and $V_{m+1}=[a_m,c_m)\cap F$ for every $m=0,1,\dots$. According to the Nested Interval Theorem, the sequence $(c_m)_{m=1}^\infty$ is convergent in $\mathbb S$ to a point $x^*\in \bigcap\limits_{m=0}^\infty V_m$. Since $F$ is closed in $\mathbb S$,  $x^*\in F$. Hence, $F\cap\bigcap\limits_{m=0}^\infty V_m\ne\emptyset$. Consequently, $F$ is $\alpha$-favorable.

  Suppose that the theorem is true for all $1\le k\le n$ and prove it for $k=n+1$.

    Consider a set  $\emptyset\ne F\subseteq \mathbb S^{n+1}$. Let the player $\beta$ chooses a set $U_0=F\cap\prod\limits_{k=1}^{n+1}[a_{0,k},b_{0,k})$ with $a_0=(a_{0,k})_{k=1}^{n+1}\in F$.
Denote $U_0^+=\prod\limits_{k=1}^{n+1}(a_{0,k},b_{0,k})$ and consider the case $U_0^+\cap F=\emptyset$. For every  $k=1,\dots, n+1$ we set $U_{0,k}=\{a_{0,k}\}\times \prod\limits_{i\ne k}[a_{0,i},b_{0,i})$ and $F_{0,k}=F\cap U_{0,k}$. Since $U_{0,k}$ is homeomorphic to $\mathbb S^n$, by the inductive assumption the space $F_{0,k}$ is $\alpha$-favorable for every $k=1,\dots,n+1$. Then $F$ is $\alpha$-favorable according to Lemma~\ref{union_of_alpha_favorable}. Now let $U_0^+\cap F\ne\emptyset$. If there exists an isolated in $\mathbb S^{n+1}$ point $x\in U_0$, then  $\alpha$ put $V_1=\{x\}$ and wins. Assume $U_0$ has no isolated points in $\mathbb S^{n+1}$. Then there is $c_0=(c_{0,k})_{k=1}^{n+1}\in U_0^+\cap F$ such that ${\rm diam}(\prod\limits_{k=1}^{n+1}[a_{0,k},c_{0,k}))<1$. We put $V_1=F\cap \prod\limits_{k=1}^{n+1}[a_{0,k},c_{0,k})$. Let $U_1=F\cap \prod\limits_{k=1}^{n+1}[a_{1,k},b_{1,k})$ be the second turn of $\beta$ such that $a_1=(a_{1,k})_{k=1}^{n+1}\in F$ and $U_1\subseteq V_1$. Again, if $U_1^+\cap F=\emptyset$, where $U_1^+=\prod\limits_{k=1}^{n+1}(a_{1,k},b_{1,k})$, then, using the inductive assumption, we obtain that for every $k=1,\dots,n+1$ the space $F\cap \bigl(\{a_{1,k}\}\times \prod\limits_{i\ne k}[a_{1,i},b_{1,i})\bigr)$ is $\alpha$-favorable. Then $\alpha$ has a winning strategy in $F$ by Lemma~\ref{union_of_alpha_favorable}. If $U_1^+\cap F\ne\emptyset$ and $U_1$ has no isolated points in $\mathbb S^{n+1}$, the player $\alpha$ chooses a point $c_1=(c_{1,k})_{k=1}^{n+1}\in U_1^+\cap F$ such that ${\rm diam}(\prod\limits_{k=1}^{n+1}[a_{1,k},c_{1,k}))<1/2$ and put $V_2=F\cap \prod\limits_{k=1}^{n+1}[a_{1,k},c_{1,k})$. Repeating this process,  we obtain sequences of points $(a_m)_{m=0}^\infty$, $(b_m)_{m=0}^\infty$ and $(c_m)_{m=0}^\infty$, and of sets $(U_m)_{m=0}^\infty$ and $(V_m)_{m=1}^\infty$, which satisfy the following properties:
  \begin{enumerate}
    \item $U_m=F\cap \prod\limits_{k=1}^{n+1}[a_{m,k},b_{m,k})$;

    \item $a_m\in F$, $c_m\in U_m^+\cap F$;

    \item $V_{m+1}=F\cap \prod\limits_{k=1}^{n+1}[a_{m,k},c_{m,k})$;

    \item $V_{m+1}\subseteq U_m\subseteq V_m$;

    \item ${\rm diam} (V_{m+1})<\frac{1}{m+1}$
  \end{enumerate}
for every $m=0,1,\dots$. We observe that the sequence $(c_m)_{m=0}^\infty$ is convergent in $\mathbb R^{n+1}$ and $x^*=\lim\limits_{m\to\infty}c_m\in \bigcap\limits_{m=0}^\infty \overline{V_m}=  \bigcap\limits_{m=0}^\infty V_m$. Since $c_m\to x^*$ in $\mathbb S^{n+1}$, $c_m\in F$ and $F$ is closed in $\mathbb S^{n+1}$,  $x^*\in F\cap \bigl(\bigcap\limits_{m=0}^\infty V_m\bigr)$. Hence, $F$ is $\alpha$-favorable.
\end{proof}

\section{Every $C^*$-embedded subspace of $\ss$ is a hereditarily Baire subspace of $\mathbb R^2$.}

\begin{lemma}\label{char_zero_set}
 A set $E\subseteq \mathbb R^2$ is functionally closed in $\mathbb S^2$ if and only if
  \begin{enumerate}
    \item\label{cond:E_is_Gdelta} $E$ is $G_\delta$ in $\mathbb R^2$; and

    \item\label{cond:sep_from_R2-closed} if $F$ is $\mathbb R^2$-closed set disjoint from $E$, then $F$ and $E$ are completely separated in~$\mathbb S^2$.
  \end{enumerate}
\end{lemma}

\begin{proof}
  {\it Necessity.} Let $f:\ss\to\mathbb R$ be a continuous function such that $E=f^{-1}(0)$. According to~\cite[Theorem 2.1]{Bade}, $f$ is a Baire-one function on~$\rr$. Consequently, $E$ is a  $G_\delta$ subset of $\rr$.

   Condition~(\ref{cond:sep_from_R2-closed}) follows from the fact that every $\mathbb R^2$-closed set is, evidently, a functionally closed subset of $\mathbb S^2$.

  {\it Sufficiency.} Since $E$ is $G_\delta$ in $\rr$, there exists a sequence of  $\mathbb R^2$-closed sets $F_n$ such that $X\setminus E=\bigcup\limits_{n=1}^\infty F_n$. Clearly, $E\cap F_n=\emptyset$. Then condition~(\ref{cond:sep_from_R2-closed}) implies that for every  $n\in\mathbb N$ there exists a continuous function $f_n:\mathbb S^2\to\mathbb R$ such that $E\subseteq f_n^{-1}(0)$ i $F_n\subseteq f^{-1}(1)$. Then $E=\bigcap\limits_{n=1}^\infty f_n^{-1}(0)$. Hence, $E$ is  functionally closed in~$\mathbb S^2$.
\end{proof}

\begin{lemma}\label{pr.3}
 Let $X$ be a metrizable space, $A\subseteq X$ be a set without isolated points and let $B\subseteq X$ be a countable set such that $A\cap B=\emptyset$. Then there exists a set $C\subseteq A$ without isolated points such that $\overline{C}\cap B=\emptyset$.
\end{lemma}

\begin{proof} Let $d$ be a metric on $X$, which generates its topological structure. For $x_0\in X$ and $r>0$ we denote $B(x_0,r)=\{x\in X:d(x,x_0)<r\}$ and $B[x_0,r]=\{x\in X:d(x,x_0)\le r\}$. Let $B=\{b_n:n\in\mathbb N\}$. We put
 $A_0=\emptyset$ and construct sequences $(A_n)_{n=1}^{\infty}$ and $(V_n)_{n=1}^{\infty}$ of nonempty finite sets $A_n\subseteq A$ and open neighborhoods $V_n$ of $b_n$ which for every $n\in\mathbb N$ satisfy the following conditions:
 \begin{gather}
A_{n-1}\subseteq A_{n};\label{eq:cond_on_sets1}\\
\forall x\in A_n \,\,\,\exists y\in A_n\setminus\{x\}\,\,\, \mbox{with } d(x,y)\leq \frac{1}{n};\label{eq:cond_on_sets2}\\
d(A_n,\bigcup_{1\leq i\leq n}V_i)>0.\label{eq:cond_on_sets3}
  \end{gather}
Let $A_1=\{x_1,y_1\}$, where $d(x_1,y_1)\le 1$ and $x_1\ne y_1$. We take $\varepsilon>0$ such that $A_1\cap B[b_1,\varepsilon]=\emptyset$ and put
$V_1=B(b_1,\varepsilon)$. Assume that we have already defined finite sets $A_1,\dots A_k$ and neighborhoods $V_1,\dots, V_k$ of $b_1,\dots,b_k$, respectively, which satisfy conditions~(\ref{eq:cond_on_sets1})--(\ref{eq:cond_on_sets3}) for every $n=1,\dots,k$. Let $A_k=\{a_1,\dots,a_{m}\}$, $m\in\mathbb N$. Taking into account that the set $D=A\setminus \bigcup\limits_{1\leq i\leq k}\overline{V}_i$ has no isolated points, for every  $i=1,\dots,m$ we take  $c_i\in D$ with $c_i\ne a_i$ and $d(a_i,c_i)\le \frac{1}{k+1}$. Put $A_{k+1}=A_k\cup \{c_1,\dots,c_m\}$. Take $\delta>0$ such that $A_{k+1}\cap B[b_{k+1},\delta]=\emptyset$. Let $V_{k+1}=B(b_{k+1},\delta)$. Repeating this process, we obtain needed sequences $(A_n)_{n=1}^\infty$ and $(V_n)_{n=1}^\infty$.

It remains to put $C=\bigcup\limits_{n=1}^\infty A_n$.
\end{proof}

The following results will be useful.

\begin{theorem}[\cite{GillJer}]\label{Gill-J}
  A subspace $E$ of a topological space  $X$ is $C^*$-embedded in $X$ if and only if every two disjoint functionally closed subsets of $E$ are completely separated in $X$.
\end{theorem}

\begin{theorem}[\cite{Terasawa}]\label{strongly_zero_dim}
  The Sorgenfrey plane $\mathbb S^2$ is strongly zero-dimensional, i.e.,  for any completely separated sets $A$ and $B$ in $\ss$  there exists a clopen set $U\subseteq \ss$ such that $A\subseteq U\subseteq \ss\setminus B$.
\end{theorem}

Recall that a space  $X$ is {\it hereditarily Baire} if every its closed subspace is Baire.

\begin{theorem}\label{th:C-emb-imply-Baire}
  Let $E$ be a $C^*$-embedded subspace of $\mathbb S^2$. Then $E$ is a hereditarily Baire subspace of $\mathbb R^2$.
\end{theorem}

\begin{proof}
Assume that $E$ is not $\mathbb R^2$-hereditarily Baire space and take an $\mathbb R^2$-closed countable subspace $E_0$ without $\mathbb  R^2$-isolated point (see~\cite{Debs}). Notice that $E$ is $\mathbb S^2$-closed according to \cite[Corollary 2.3]{Tanaka}. Therefore, $E_0$ is $\mathbb S^2$-closed set. By Theorem~\ref{th:Sorgenfey-alpha-favorable} the space $E_0$ is $\alpha$-favorable, and, consequently, $E_0$ is a Baire subspace of $\ss$.

Let $E_0'$ be a set of all $\mathbb S^2$-nonisolated points of $E_0$. Since $E_0'$ is the set of the first category in $\mathbb S^2$-Baire space
  $E_0$, the set $G=E_0\setminus E_0'$ is $\mathbb S^2$-dense open discrete subspace of $E_0$. We notice that $G$ is $\mathbb  R^2$-dense subspace of $E_0$. By Lemma~\ref{pr.3} there exists a set $C\subseteq G$ without $\mathbb R^2$-isolated point such that $\crr C\cap
  E_0'=\emptyset$. We put $F=\crr C\cap E_0$.

  Let $A$ and $B$ be any $\mathbb R^2$-dense in $F$ disjoint sets such that $F=A\cup B$. Evidently $A$ and $B$ are clopen subsets of $F$, since $F$ is $\mathbb S^2$-discrete space. Notice that $F$ is  $z$-embedded in $E$, because $F$ is countable.  Moreover, $F$ is $\mathbb
  R^2$-closed in $E$. Hence,  $F$ is $\mathbb S^2$-functionally closed in $E$. By Theorem~\ref{th:Blair-Hager} the set $F$ is $C$-embedded in $C^*$-embedded in $\ss$ set $E$. Consequently, $F$ is $C^*$-embedded in $\mathbb S^2$. Therefore, Theorem~\ref{Gill-J} and Theorem~\ref{strongly_zero_dim} imply that there exist disjoint clopen set $U,V\subseteq \mathbb S^2$ such that $A=U\cap F$
  and $B=V\cap F$. According to Lemma~\ref{char_zero_set} the sets $U$ and $V$ are $G_\delta$ in $\mathbb R^2$. Let $D=\crr F$. Then $U\cap D$ and $V\cap D$ are  $\mathbb R^2$-dense in $D$ disjoint  $G_\delta$-sets, which contradicts to the baireness of $D$.
\end{proof}

\section{Every discrete $C^*$-embedded subspace of $\ss$ is a countable $G_\delta$-subspace of~$\rr$.}

\begin{lemma}\label{lemma:uncount} Let $X$ be a metrizable separable space and $A\subseteq X$ be an uncountable set. Then there exists a set $Q\subseteq A$ which is homeomorphic to the set $\mathbb Q$ of all rational numbers.
\end{lemma}

\begin{proof}
  Let $A_0$ be the set of all points of $A$ which are not condensation points $A$ (a point $a\in X$ is called {\it a condensation point of $A$ in
  $X$} if every neighborhood of $a$ contains uncountably many elements of $A$). Notice that  $A_0$ is countable, since $X$ has a countable base. Put  $B=A\setminus A_0$. Then the inequality $|A|>\aleph_0$ implies that every point of  $B$ is a condensation point of $B$. Take a countable subset  $Q\subseteq B$ which is dense in $B$. Clearly, every point of $Q$ is not isolated. Hence,  $Q$ is homeomorphic to $\mathbb Q$ by the Sierpi\'{n}ski Theorem~\cite{Serp}.
\end{proof}

\begin{lemma}\label{lemma:z-emb-imply-isol}
  Let $E$ be an $\mathbb R^2$-hereditarily Baire $z$-embedded subspace of $\mathbb S^2$. Then the set $E^0$ of all isolated points of  $E$ is at most countable.
\end{lemma}

\begin{proof}
   Assume $E^0$ is uncountable. Notice that $E^0$ is an $F_\sigma$-subset of $E$, since $E^0$ is an open subset of $E$ and $\mathbb   S^2$ is a perfect space by~\cite{HM}. Then $E^0=\bigcup\limits_{n=1}^\infty E_n$, where every set  $E_n$ is closed in $E$. Take  $N\in\mathbb N$ such that  $E_N$ is uncountable. According to Lemma~\ref{lemma:uncount} there exists a set $Q\subseteq E_N$ which is homeomorphic to $\mathbb Q$. Since $Q$ is clopen in  $E_N$ and $E_N$ is a clopen subset of a $z$-embedded in $\ss$ set $E$, there exists a functionally closed subset  $Q_1$ of $\mathbb S^2$ such that   $Q=E\cap Q_1$. By Lemma~\ref{char_zero_set} the set $Q_1$ is a $G_\delta$-set in $\mathbb R^2$. Then $Q$ is a $G_\delta$-subset of a hereditarily Baire space $E$. Hence, $Q$ is a Baire space, a contradiction.
\end{proof}

\begin{theorem}\label{cor:discr1}
  If $E$ is a discrete $C^*$-embedded subspace of $\ss$, then $E$ is a countable $G_\delta$-subspace of~$\rr$.
\end{theorem}

\begin{proof}
  Theorem~\ref{th:C-emb-imply-Baire} and Lemma~\ref{lemma:z-emb-imply-isol} imply that $E$ is a countable hereditarily Baire subspace of $\rr$. According to~\cite[Proposition 12]{Ka} the set $E$ is  $G_\delta$ in $\rr$.
\end{proof}

The converse implication in Theorem~\ref{cor:discr1} is not valid as Theorem~\ref{ex:converse_to_countable} shows.

\begin{lemma}\label{lemma:closed}
  Let $A$ be an $\mathbb S^2$-closed set, $\varepsilon>0$ and $L(A;\varepsilon)=\{p\in \mathbb S^2: B[p;\varepsilon)\subseteq A\}$. Then $L(A;\varepsilon)$ is $\mathbb R^2$-closed.
\end{lemma}

\begin{proof}
We take $p_0=(x_0,y_0)\in {\rm cl}_{\mathbb R^2} L(A;\varepsilon)$ and show that $p_0\in L(A;\varepsilon)$. We consider
$U={\rm int}_{\mathbb R^2} B[p_0;\varepsilon)$ and prove that $U\subseteq A$. Take $p=(x,y)\in U$ and  put
$\delta=\min\{(x-x_0)/2, (y-y_0)/2,(x_0+\varepsilon-x)/2,(y_0+\varepsilon-y)/2\}$. Let
 $p_1\in B(p_0;\delta)\cap L(A;\varepsilon)$. It is easy to see that $p\in B[p_1;\varepsilon)$. Then $p\in A$, since $p_1\in L(A;\varepsilon)$. Hence, $U\subseteq A$. Then $B[p_0;\varepsilon)={\rm cl}_{\mathbb S^2}U\subseteq {\rm cl}_{\mathbb S^2}A=A$, which implies that $p_0\in L(A;\varepsilon)$. Therefore, $L(A;\varepsilon)$ is closed in $\rr$.
\end{proof}

\begin{theorem}\label{ex:converse_to_countable}
  There exists an $\ss$-closed countable discrete $G_\delta$-subspace $E$ of~$\rr$ which is not $C^*$-embedded in~$\ss$.
\end{theorem}

\begin{proof} Let $C$ be the standard Cantor set on $[0,1]$ and let  $(I_n)_{n=1}^\infty$ be a sequence of all complementary intervals $I_n=(a_n,b_n)$ to $C$ such that ${\rm diam\,} (I_{n+1})\le {\rm diam\,} (I_n)$ for every $n\ge 1$. We put $p_n=(b_n;1-a_n)$,  $E=\{p_n:n\in\mathbb N\}$ and $F=\{(x,1-x):x\in\mathbb R\}\cap (C\times [0,1])$. Notice that $E$ is a closed subset of $\ss$,  $F$ is functionally closed in $\ss$ and $E\cap F=\emptyset$.

Let $N'\subseteq\mathbb N$ be a set such that $\{b_n:n\in N'\}$ and $\{b_n:n\in\mathbb N\setminus N'\}$ are dense subsets of $C$. To show that $E$ is not $C^*$-embedded in $\ss$ we verify that disjoint clopen subsets
$$
E_1=\{p_n:n\in N'\}\quad\mbox{and}\quad  E_2=\{p_n:n\in\mathbb N\setminus N'\}
$$
of $E$ can not be separated by disjoint clopen subsets in $\ss$. Assume the contrary and  take disjoint clopen subsets $W_1$ and $W_2$ of $\ss$ such that $W_i\cap E=E_i$ for $i=1,2$.

We prove that $W_1\cap F$ is $\rr$-dense in $F$. To obtain a contradiction we take an $\rr$-open set $O$ such that $O\cap F\cap W_1=\emptyset$.  Since the set $U=\ss\setminus W_1$ is clopen, $U=\bigcup\limits_{n=1}^\infty L(U;\frac 1n)$, where $L(U;\frac 1n)=\{p\in \mathbb S^2: B[p;1/n)\subseteq U\}$ and the set $F_n=L(U;\frac 1n)$ is  $\mathbb R^2$-closed by Lemma~\ref{lemma:closed} for every $n\in\mathbb N$. Since $O\cap F$ is a Baire subspace of $\mathbb R^2$, there exist  $N\in\mathbb N$ and an $\rr$-open in $F$ subset $I\subseteq F$ such that $I\cap O\subseteq F_{N}\cap F\subseteq \ss\setminus E_1$. Taking into account that ${\rm diam}\, (I_n)\to 0$, we choose  $n_1>N$ such that $b_n-a_n<\frac {1}{2N}$ for all $n\ge n_1$. Since the set $\{a_n:n\in N'\}$ is dense in $C$, there exists $n_2\in N'$ such that $n_2>n_1$ and $p=(a_{n_2};1-a_{n_2})\in I$. Clearly, $p\in F$. Consequently, $B[p;\frac 1N)\cap E_1=\emptyset$. But $p_{n_2}\in B[p,\frac 1N)\cap E_1$, a contradiction.

Similarly we can show that $W_2\cap F$ is also $\rr$-dense in $F$.

Notice that $W_1$ and $W_2$ are $G_\delta$ in $\rr$ by Lemma~\ref{char_zero_set}. Hence, $W_1\cap F$ and $W_2\cap F$ are disjoint  dense $G_\delta$-subsets of a Baire space $F$, which implies a contradiction. Therefore, $E$ is not $C^*$-embedded in $\ss$.
\end{proof}

\section{A characterization of $C$-embedded subsets of the anti-diagonal of $\mathbb S^2$.}

By $\mathbb D$ we denote the {\it anti-diagonal} $\{(x,-x):x\in\mathbb R\}$ of the Sorgenfrey plane. Notice that $\mathbb D$ is a closed discrete subspace of $\ss$.

\begin{theorem}
 For a set $E\subseteq {\mathbb D}$ the following conditions are equivalent:
  \begin{enumerate}
   \item $E$ is $C$-embedded in $\mathbb S^2$;

   \item $E$ is $C^*$-embedded in $\mathbb S^2$;

   \item $E$ is a countable $G_\delta$-subspace of $\rr$;

   \item $E$ is a countable functionally closed subspace of $\ss$.
  \end{enumerate}
\end{theorem}

\begin{proof}
The implication $(1)\Rightarrow (2)$ is obvious.
The implication $(2)\Rightarrow (3)$ follows from Theorem~\ref{cor:discr1}.

We prove  $(3)\Rightarrow (4)$. To do this we verify condition~(\ref{cond:sep_from_R2-closed}) from Lemma~\ref{char_zero_set}. Let $F$  be an
$\rr$-closed set disjoint from $E$. Denote $D=F\cap \mathbb D$ and  $U=\bigcup\limits_{p\in D}B[p;1)$. We show that  $U$ is clopen in $\ss$. Clearly, $U$ is open in $\ss$. Take a point $p_0\in{\rm cl}_{\ss}U$ and show that  $p_0\in U$. Choose a sequence $p_n\in U$ such that $p_n\to p_0$ in $\ss$. For every $n$ there exists $q_n\in D$ such that $p_n\in B[q_n,1)$. Notice that the sequence $(q_n)_{n=1}^\infty$ is bounded in  $\rr$ and take a convergent in  $\rr$ subsequence $(q_{n_k})_{k=1}^\infty$ of $(q_n)_{n=1}^\infty$. Since  $D$ is $\rr$-closed, $q_0=\lim\limits_{k\to\infty}q_{n_k}\in D$. Then  $p_0\in{\rm cl}_{\rr}B[q_0,1)$. If $p_0\in B[q_0,1)$, then $p_0\in U$. Assume $p_0\not\in B[q_0,1)$ and let $q_0=(x_0,y_0)$. Without loss of generality we may suppose that $p_0\in [x_0,x_0+1]\times\{y_0+1\}$. Since  $p_{n_k}\to p_0$ in $\ss$,  $q_{n_k}\in (-\infty,x_0]\times [y_0,+\infty)$ for all $k\ge k_0$ and $p_0\in [x_0,x_0+1)\times\{y_0+1\}$. Then $p_0\in \bigcup_{k=1}^\infty B[q_{n_k},1)\subseteq U$. Hence, $U$ is clopen and  $D=U\cap\mathbb D$. Since $\mathbb D$ and $F\setminus U$ are disjoint functionally closed subsets of $\ss$, there exists a clopen set  $V$ such that  $\mathbb D\cap V=\emptyset$ and $F\setminus U\subseteq V$. Then $F\subseteq U\cup V\subseteq \ss\setminus E$. Consequently, $F$ and $E$ are completely separated in $\ss$. Therefore, $E$ is functionally closed in $\ss$ by Lemma~\ref{char_zero_set}.

$(4)\Rightarrow (1)$.
Notice that  $E$ satisfy the conditions of Theorem~\ref{th:Blair-Hager}. Indeed, $E$ is $z$-embedded in $\ss$, since $|E|\le\aleph_0$. Moreover, $E$ is well-embedded in $\ss$, since $E$ is functionally closed.
\end{proof}

\begin{remark}
  {\rm Notice that a subset $E$ of $\rr$ is countable $G_\delta$ if and only if it is scattered in $\rr$. Indeed, assume that $E$ is countable $G_\delta$-set which contains a set $Q$ without isolated points. Then $Q$ is a $G_\delta$-subset of $\rr$ which is homeomorphic to $\mathbb Q$, a contradiction. On the other hand, if $E$ is scattered, then Lemma~\ref{lemma:uncount} implies that $E$ is countable. Since $E$ is hereditarily Baire and countable, $E$ is $G_\delta$ in $\rr$.}
\end{remark}

Finally, we show that the Sorgenfrey plane is not a $\delta$-normally separated space. Let $E=\{(x,-x):x\in\mathbb Q\}$ and $F=\mathbb D\setminus E$. Then $E$ is closed and $F$ is functionally closed in $\ss$, since $F$ is the difference of the functionally closed set $\mathbb D$ and the functionally open set $\bigcup\limits_{p\in E}B[p,1)$. But $E$ and $F$ can not be separated by disjoint clopen sets in $\ss$, because $E$ is not  $G_\delta$-subset of $\mathbb D$ in $\rr$.

\end{document}